\documentclass{article}
\usepackage{epsfig}
\usepackage{graphics}
\usepackage{amsmath}
\usepackage{amssymb}
\usepackage{amsfonts}
\usepackage{amsthm}
\usepackage{mathrsfs}
\usepackage{hyperref}
\usepackage{algorithm}
\usepackage{algorithmic}
\usepackage{blindtext}
\usepackage[utf8]{inputenc}
\usepackage{graphicx}
\usepackage{subfigure}
\usepackage[usenames, dvipsnames]{color}
\usepackage{rotating} 

\usepackage{adjustbox}

\usepackage{pgfplots,tikz}
\usetikzlibrary{plotmarks}

\newdimen\prevdp
\def\leftlabel#1{\noalign{\prevdp=\prevdepth
		\kern-\prevdp\nointerlineskip\vbox to0pt{\vss\hbox{#1}}\kern\prevdp}}

\newcommand {\F}        {{\mathbb{F}}}
\newcommand {\C}        {{\mathbb{C}}}
\newcommand {\R}        {{\mathbb{R}}}

\newcommand{\norm}[1]{\left\Vert#1\right\Vert}

 \newtheorem{theorem}{Theorem}[section]
 \newtheorem{lemma}[theorem]{Lemma}

  \numberwithin{equation}{section}
 \newtheorem{assumption}[theorem]{Assumption}
 \newtheorem{remark}[theorem]{Remark}

\newcommand {\cH}  {\mathcal{H}}

\newcommand {\cV}  {\mathcal{V}}

\newcommand {\Hinf}	{{\mathcal H}_{\infty}}

\newcommand {\ri}	{{\mathrm i}}

\DeclareMathOperator{\Col}{Col}

\DeclareMathOperator{\rank}{rank}
\DeclareMathOperator{\Real}{Re}
\DeclareMathOperator{\Imag}{Im}

\title{Subspace Method for the Estimation of Large-Scale Structured 
Real Stability Radius} 
\author{
Nicat Aliyev\thanks{Charles University, Faculty of Mathematics and Physics, Department of Numerical Mathematics, Sokolovska 83, 18675 Praha, Czech Republic.  E-Mail: \texttt{aliyev@karlin.mff.cuni.cz}. This work was supported by OP RDE, project No. CZ.02.2.69/0.0/0.0/18$\_$053/0016976. International mobility of research, technical and administrative staff at Charles University.} 
}

\begin{document}
\maketitle

\begin{abstract}

\noindent 
We consider the autonomous dynamical system 
$x' = Ax,$
with $A\in\R^{n\times n}.$ 
This linear dynamical system is said to be \textit{asymptotically stable} if all of the eigenvalues of $A$ lie in the open left-half of the complex plane. In this case, the matrix $A$ is said to be \textit{ Hurwitz stable} or shortly a \textit{stable} matrix. In practice, stability of a system can be violated because of arbitrarily small perturbations such as modeling errors. In such cases, one deals with the robust stability of the system rather than its stability. The system above is said to be\textit{ robustly stable} if the system, as well as all of its arbitrarily small perturbations, are stable. To measure the robustness of the system subject to perturbations, a quantity of interest is the \textit{stability radius} or in other words \textit{distance to instability}. 
In this paper we focus on the estimation of the structured 
real stability radius for large-scale systems. We propose a subspace framework to estimate the structured real stability radius and prove that our new method converges at a quadratic rate in theory. Our method benefits from a one-sided interpolatory model order reduction technique, in a sense that the left and the right subspaces are the same. 
The quadratic convergence of the method is due to the certain Hermite interpolation properties between the full and reduced problems. The proposed framework estimate the structured 
real stability radius 
for large-scale systems efficiently. The efficiency of the method is demonstrated on several numerical experiments.

\vskip 1ex


\vskip 1ex

\end{abstract}

\section{Introduction}
Consider the autonomous dynamical system 
\begin{equation}\label{eq.sys}
x' = Ax,
\end{equation} 
with $A\in\F^{n\times n}$ and $\F = \R$ or $\F = \C.$
One important property of this linear dynamical system is the notion of \textit{stability.} Formally, system \eqref{eq.sys} is said to be \textit{asymptotically stable} if all of the eigenvalues of $A$ lie in the open left-half of the complex plane, i.e., $\Lambda(A)\subset \C^- $, where $\Lambda(A)$ denotes the spectrum of $A,$ $\C^- := \left\{z \in \C \; :\;  \Real(z) < 0\right\}$ and $\Real(\cdot)$ stands for the real part of its argument \cite{Dullerud2000}. In this case, the matrix $A$ is said to be \textit{ Hurwitz stable} or shortly a \textit{stable} matrix. In practice, arbitrarily small perturbations may cause the stable system to become unstable. In other words, some or all of the eigenvalues of $A$ may be moved into the right-half plane by applying small perturbations. In such cases, robust stability of the system becomes much more important. The system \eqref{eq.sys} is said to be\textit{ robustly stable} if the system, as well as all of its arbitrarily small perturbations, are stable. 
To detect such cases and to measure the robustness of the system subject to perturbations, a quantity of interest is the \textit{stability radius} or in other words \textit{distance to instability}. The unstructured and structured stability radii have been introduced in \cite{Hinrichsen1986a, Hinrichsen1986b, Hinrichsen1990, Vanloan1985}. 
The structured stability radius of a matrix triple $(A,B,C)\in\F^{n\times n} \times \F^{n\times m} \times \F^{p\times n} $ is defined by 
\begin{equation}\label{eq:str}
r_{\F}(A;B,C): = \inf\left\{ \Vert \Delta \Vert_2 \; : \;  \Delta\in\F^{m\times p} \; \text{and} \;  A + B\Delta C \; \text{is unstable} \right\} 
\end{equation}
\noindent where $\Delta$ is a perturbation (disturbance) matrix, $B\in\F^{n\times m} $ and $C\in\F^{p\times n}$ are restriction matrices that determine the structure of the perturbation. For example, the matrices $B$ and $C$ may reflect the possibility that only certain elements of $A$ are subject to perturbations \cite{Hinrichsen1986a}.
When $B = I $ and $C = I,$ we abbreviate $r_F(A;I,I)$ by $r_\F(A)$ and call it the unstructured stability radius of $A.$  
In many applications (see e.g., \cite{Trefethen2005} ), it is required to consider only a certain class of perturbations. For example, when $A$ is a real matrix, then to allow only real perturbations is more plausible. In this paper, we turn our attention into the real structured 
perturbations only. For real $(A,B,C)$, $r_{\R}(A;B,C)$ and $r_{\R}(A)$ are called the real structured and unstructured stability radii, respectively. 

Let us denote the singular values of $M\in\F^{p\times m},$ ordered non-increasingly, by 
$\sigma_{i}(M),$ $i = 1,2, \ldots, \min\left\{p,m \right\},$ and denote the real and imaginary parts of $M$ by $\Real(M)$ and $\Imag(M),$ respectively. 
A formula for $r_\R(A; B,C)$ given by Qiu $et \,al.$ \cite{Qiu1995} in the form of an optimization problem in two variables is as follows:
\begin{equation}\label{eq:basic}
r_{\R}(A;B,C)^{-1} = \sup_{\omega\in\R} \inf_{\gamma\in (0,1]}\sigma_{2}\left(
\begin{bmatrix}
\Real( C(\ri\omega I-A)^{-1}B ) &  -\gamma \Imag(C(\ri\omega I-A)^{-1}B) \\
\frac{1}{\gamma} \Imag(C(\ri\omega I-A)^{-1}B) & \Real(C(\ri\omega I-A)^{-1}B)
\end{bmatrix} 
\right).
\end{equation}
Setting
\begin{equation} 
H(s):= C(sI-A)^{-1}B  \label{eq:trans} 
\end{equation}
and
\begin{equation}
\mu(M):= \inf_{\gamma\in (0,1]}\sigma_{2}\left(
\begin{bmatrix}
\Real(M) & -\gamma \Imag(M) \\
\frac{1}{\gamma} \Imag(M) & \Real(M)
\end{bmatrix}\right) \label{eq:mu},
\end{equation}
the problem in \eqref{eq:basic} takes the form of
\begin{equation} \label{eq:supmu}
r_{\R}(A;B,C) = \left\{ \sup_{\omega\in\R}\mu \left(H(\ri\omega)\right) \right\}^{-1}.
\end{equation}
Computationally more plausible characterization for the unstructured real stability radius $r_\R(A)$ given by Qiu $et \, al.$ is as follows:
\begin{equation}
r_{\R}(A) := \min_{\omega\in\R} \max_{\gamma\in (0,1]}\sigma_{2n-1}\left(
\begin{bmatrix}
A & -\omega \gamma I \\
\frac{\omega}{\gamma} & A
\end{bmatrix}\right).
\end{equation}
It was proved in \cite{Qiu1995} that, the function to be minimized in \eqref{eq:mu} is unimodal on $(0, 1].$ Therefore, any local minimiser is actually a global minimiser, and the computation of $\mu(M)$ for a given $M$ is the easier part of the problem. In fact, any search algorithm can be used to solve the inner minimization problem \eqref{eq:mu} with the guarantee of global convergence. But the outer maximization problem in \eqref{eq:supmu} is still challenging, especially when $n$ is large.

\subsection{Literature Review} Several algorithms have been proposed so far to estimate the structured and unstructured real stability radii given by the characterizations above. Based on the well known correspondence between the singular values of the transfer function and the imaginary eigenvalues of certain Hamiltonian matrices, Sredhar $et. al$ proposed an iterative method \cite{Sredhar1996}. Although numerical experiments on their work suggest that the rate of the convergence of their algorithm is quadratic, a rigorous proof does not appear in the paper. Repeated solution of Hamiltonian eigenvalue problems of size twice the order of the system, makes this method inconvenient for large-scale problems. Freitag and Spence \cite{Freitag2014b} focus on the unstructured case and they propose an implicit determinant method using Newton’s method. Their method also benefits from the link between the singular values of the transfer function and the imaginary eigenvalues of the Hamiltonian matrices. To find the critical point corresponding to the desired singular value, they impose the implicit determinant method. The advantage of their method is that, instead of solving singular value problems and Hamiltonian eigenvalue problems, they solve linear systems at each Newton step. This makes their method more efficient and convenient for large-scale systems.
Another method to handle the large-scale problems is proposed by Guglielmi and Manetta \cite{Guglielmi2014b}.
Their method consists of the inner and outer iterations and does not rely on the characterization \eqref{eq:basic}. The inner iteration approximates the $\epsilon-$ pseudospectral abscissa from the above for a fixed $\epsilon$, whereas the outer iteration varies $\epsilon$ by means of the Newton iterations. Their method converges locally and provides only upper bounds for the stability radius. 
More recent works for the structured real stability radius of sparse systems are \cite{Katewa2020, Guglielmi2017}. But the distance measure involved in these works is Frobenius norm. 
Katewa and Pasqualetti \cite{Katewa2020} formulate the stability radius problem as an equality-constrained optimization problem and by means of Lagrangian method, they characterize the optimality conditions by revealing the relation between an optimal perturbation and the eigenvectors of an optimally perturbed system. Consequently, they develop a gradient descent algorithm that converges locally. The method in \cite{Guglielmi2017} is based on the relationship between the spectral value set abscissa and the real stability radius. 

In this work, we are concerned with the computation of the structured 
real stability radius for large systems. 
We assume that the number columns of the restriction matrix $B$ and the number of rows of the restriction matrix $C$ are relatively small, i.e., $n\gg m,p,$ which is usually the case in practice. 
We propose a subspace framework to estimate $r_\R(A;B,C).$
For this, we adapt the technique of our recent works \cite{Aliyev2017, Aliyev2020}, which are devised to compute and to minimize the $\Hinf$ norm of large-scale systems, respectively.
The proposed method converges fast with respect to the subspace dimension. This fast convergence is verified by a rigorous analysis and observed by means of several numerical examples. 

The rest of this work is organized as follows. Next section describes a subspace framework for the structured real stability radius. The method is based on an interpolatory model order reduction technique. In section \ref{sec:roc} we give a rigorous rate of convergence analysis of the subspace framework described in the next section. Section \ref{sec:nr} is devoted to numerical experiments that illustrate fast convergence of the algorithm presented in this work. 

\section{Subspace Framework for $r_{\R}(A; B, C)$}\label{sec:subspace}
In this section we present a subspace framework for the estimation of the structured real stability radius $r_{\R}(A; B, C)$ for a large-scale stable matrix $A\in\R^{n \times n}$ and the given restriction matrices $B \in\R^{n \times m}, C\in\R^{p\times n}.$
The matrix valued function $H(s) = C(sI - A)^{-1}B$ in \eqref{eq:trans}
corresponds to the transfer function of the standard linear time-invariant system 
\begin{align}\label{eq.desc}
x' &= Ax + Bu \\ \nonumber
y &= Cx.
\end{align}
with $A\in\R^{n\times n}, B\in\R^{n\times m}$ and $C\in\R^{p\times n}$
whose $\cH_\infty$-norm is defined by
\begin{equation*}
 \left\| H \right\|_{\cH_\infty} := \sup_{s \in \C^+} \left\| H(s) \right\|_2 = \sup_{s \in \partial\C^+} \left\| H(s) \right\|_2 = \sup_{\omega \in \R} \sigma_1(H(\ri \omega)). 
\end{equation*}
Inspired by \cite{Kangal2018}, in our recent works \cite{Aliyev2017, Aliyev2020}, we proposed subspace frameworks by employing two-sided projections which aim at approximating and minimizing the $\cH_\infty$ norm of large-scale control systems. The main idea of this section is inspired by those works. One significant difference is that here the second largest singular value function is involved in a maximin problem, whereas in \cite{Aliyev2017} we consider the maximization of the largest singular value function and in \cite{Aliyev2020} the largest singular value function is involved in a minimax problem 
Another and more challenging difference is that, here the real and imaginary parts of the transfer function $H$ appears as block components of the block matrix 
$$ T(\omega, \gamma): = \begin{bmatrix}
\Real(H(\ri\omega)) & -\gamma \Imag(H(\ri\omega)) \\
\frac{1}{\gamma} \Imag(H(\ri\omega)) & \Real(H(\ri\omega))
\end{bmatrix}.$$
The latter difference motivates us to construct the subspace framework in a slightly different way from the previous works. Indeed, here we resort to one-sided projection in a sense that the left and the right subspaces are the same. Furthermore, we form the subspace in a way so that the algorithm we discuss here converges quadratically.
Since the large scale nature of the problem in the characterization \eqref{eq:basic} is hidden in the middle factor of $C(sI-\omega)^{-1}B,$ when $m,p$ are relatively small, we focus on the reduction of the middle factor $D(s):=(sI-A)$ of $H(s)$ to a much smaller dimension using one-sided projection. To this end, we determine a subspace $\cV\subset \C^n$ of small dimension and a matrix $V$ whose columns form an orthonormal basis for $\cV.$ 
Then the projected reduced problem is defined in terms of the matrices
$$
A^V = V^{*}AV, \quad B^V = V^*B, \quad C^V = CV.
$$
More precisely, the reduced problem is defined by
\begin{equation}\label{eq.outer}
r^\cV(A;B,C) := \left\{ \sup_{\omega\in\R}\mu(H^\cV(\ri\omega) \right\}^{-1}
\end{equation}
where  
\begin{equation}\label{eq:red_trans}
H^\cV(s): =C^V(sI-A^V)^{-1}B^V.
\end{equation}
In our subspace framework, the main objective is to get Hermite interpolation properties between the original problem defined in terms of $H(s)$ and the reduced problem defined in terms of $H^\cV(s)$, which in turn give rise to a quadratic convergence. Formally, for a given $\omega\in\R,$ we aim to form the subspace $\cV$ in a way so that, we have
\begin{align}
\mu(H(\ri\omega)) &= \mu\left(H^\cV(\ri\omega) \right),  
\label{mu:hermite1} \\
\mu'\left(H(\ri\omega)\right) &= 
\mu'\left( H^\cV(\ri\omega) \right) \label{mu:hermite2} 
\quad \text{and} \\
\mu''\left(H(\ri\omega)\right)& =  \mu''\left( H^\cV(\ri\omega) \right) 
\label{mu:hermite3}  
\end{align}
The following theorem, a particular instance of \cite[Theorem~1]{Gugercin2009}, is helpful for this purpose.
\begin{theorem}\label{thm:Gal_interpolate}
Let $H(s)$ and $H^\cV(s) $ be defined as in (\ref{eq:trans}) and \eqref{eq:red_trans}, respectively.
For given $\widehat{s} \in {\mathbb C}$, $\: \widehat{b} \in {\mathbb C}^m$ and a positive integer $N$, if
 	\begin{equation}\label{eq:subspace_incl}
		\left[ (\widehat{s}I - A)^{-1} \right]^\ell B \widehat{b}	\; \in \; {\mathcal V}		\quad \quad {\rm for} \;\;	\ell = 1, \dots, N,
 	\end{equation}
and $V$ has orthonormal columns, that is, $V^* V = I$, 
then, denoting the $\ell$-th derivatives of
	$H(s)$ and $H^\cV(s)$ at the point $\widehat{s}$ with $H^{(\ell)}(\widehat{s})$ and $\left[H^{\mathcal V}\right]^{(\ell)}(\widehat{s})$,
	we have
	\begin{equation}\label{eq:tangent_interpolate}
	 	H^{(\ell)} (\widehat{s}) \widehat{b}	=	\big[ H^{\mathcal{V} } \big]^{(\ell)} (\widehat{s}) \widehat{b}	\quad\quad {\rm for} \;\;	\ell = 0, \dots, N-1
	 \end{equation}
provided that both \, $\widehat{s} I - A$ and\, $\widehat{s} I - A^V$ are invertible.
%
 %
 %
%
%
\end{theorem}
Theorem \ref{thm:Gal_interpolate} yields us a direction that indicates how to construct the subspace $\cV$ so that the Hermite interpolation properties \eqref{mu:hermite1}-\eqref{mu:hermite3} are satisfied between
the original and the reduced problems.
For a given $\widetilde{\omega}\in\R,$ setting
$$
\cV:= \Col \left(
\begin{bmatrix}
(\ri \widetilde{\omega} I -A)^{-1}B & (\ri\widetilde{\omega} I - A)^{-2}B & (\ri\widetilde{\omega} I - A)^{-3}B
 \end{bmatrix}
 \right)
 $$
where $\Col(M)$ denotes the column space of the matrix $M,$ we conclude from Theorem \ref{thm:Gal_interpolate}  that
$$
H(\ri\widetilde{\omega}) = H^{\cV}(\ri\widetilde{\omega}), \quad   H'(\ri{\widetilde{\omega}}) =\left[H^{\cV}\right]'(\ri{\widetilde{\omega}})
\quad \text{and} \quad
H''(\ri{\widetilde{\omega}}) =\left[H^{\cV}\right]''(\ri{\widetilde{\omega}})
$$
which in turn imply
\begin{eqnarray}
T(\widetilde{\omega}, \gamma) &=& T^\cV(\widetilde{\omega},\gamma),  \label{eq:Thermite1} \\
\frac {\partial T}{\partial \omega} \left( \widetilde{\omega}, \gamma \right) 
&=& \frac {\partial T^\cV}{\partial \omega} \left( \widetilde{\omega}, \gamma \right) \label{eq:Thermite2} \\
\frac {\partial^2 T}{\partial \omega^2} \left( \widetilde{\omega}, \gamma \right) 
&=& \frac{ \partial^2 T^\cV }{\partial \omega^2} \left( \widetilde{\omega}, \gamma \right) 
\label{eq:Thermite3}
\end{eqnarray}
for all $\gamma \in(0,1],$ where
$$ T^\cV(\omega, \gamma): = \begin{bmatrix}
\Real(H^\cV(\ri \omega) & -\gamma \Imag(H^\cV(\ri\omega)) \\
\frac{1}{\gamma} \Imag(H^\cV(\ri\omega)) & \Real(H^\cV(\ri\omega))
\end{bmatrix}.$$
The equalities \eqref{eq:Thermite1}-\eqref{eq:Thermite3} give rise to the Hermite interpolation properties between 
the reduced and original problems. We will discuss this in detail in Theorem \ref{thm:basic_int}.

The subspace framework for the computation of $r_{\R}(A;B,C)$ is summarized in Algorithm \ref{alg1}, where and throughout the rest of this work, we use the short-hand notations
\[
	 \sigma_i(\omega,\gamma) := \sigma_i(T( \omega,\gamma)) \quad \text{and} \quad  \sigma_i^{{\mathcal V}}(\omega,\gamma) := \sigma_i\left( T^{{\mathcal V}}(\omega,\gamma) \right) \quad \quad 
\]
for $i = 1,\ldots, \min\left\{m,p \right\}.$
We also use the notations
\[
	 \mu(\omega) := \mu(H(\ri \omega)) \quad \text{and} \quad  \mu^{{\mathcal V}}(\omega) := \mu\left( H^{{\mathcal V}}(\ri\omega) \right).
\]
Moreover, 
$\operatorname{orth}(M)$ stands for a matrix whose columns form an orthonormal basis for the column space of the matrix $M.$
Finally, we define the unique minimizers of of the inner iterations of the original and reduced problems by
\begin{equation}\label{eq:inner_opt}
\gamma(\omega):= \arg\min_{\gamma\in(0,1]}\sigma_2(\omega,\gamma)
\quad \quad \text{and} \quad \quad 
\gamma^{\cV_k}(\omega):= \arg\min_{\gamma\in(0,1]}\sigma_2^{\cV_k}(\omega,\gamma).
\end{equation}

\begin{algorithm}[h]
 \begin{algorithmic}[1]
\REQUIRE{Matrices $A\in\R^{n\times n}, \; B\in\R^{n\times m},\;  C\in\R^{p\times n}$.}

\ENSURE{Sequence $ \{ \mu_k \}$ s.t. $1/{\mu_k}$ for sufficiently large $k$ is an estimate for $r_{\R}(A; B,C)$.}

 \STATE Choose initial interpolation point $\omega_{1}\in\R$  

\STATE $V_1 \gets  \operatorname{orth} \left( \begin{bmatrix} (\ri\omega_1 I -A)^{-1}B  \; & (\ri\omega_1 I -A)^{-2}B \;
 & (\ri\omega_1 I -A)^{-3}B \;
									 	\end{bmatrix} \right)$.  \label{init_subspaces}

\FOR{$k = 1,\,2,\,\dots$}
	\STATE
	$\omega_{k+1} \gets  \arg \max_{\omega \in \R} \mu^{{\mathcal V}_k}\left(\omega\right)$. \label{line:opt_omega}
	\STATE $\widetilde{V}_{k+1} \gets  \begin{bmatrix} (\ri\omega_{k+1}I-A)^{-1}B \; & (\ri\omega_{k+1}I-A)^{-2}B \; & (\ri\omega_{k+1}I-A)^{-3}B \; 
	 \end{bmatrix}$.
	\label{later_subspaces}
%
	\STATE $V_{k+1} \gets \operatorname{orth}\left(\begin{bmatrix} V_{k} \; & \widetilde{V}_{k+1} \end{bmatrix}\right)$ and 
	${\mathcal V}_{k+1} \gets \Col\big(V_{k+1}\big)$. \label{orth_subspaces}

	\STATE $\mu_{k+1} \gets  \mu^{ \cV_{k+1}}(\omega_{k+1})$.
\ENDFOR
\end{algorithmic}
\caption{$\;$ Subspace Method for Large-Scale Estimation of the  Structured Real Stability Radius $r_{\R}(A; B, C)$.}
\label{alg1}
\end{algorithm}

At each iteration of Algorithm \ref{alg1}, a reduced problem is solved 
in line \ref{line:opt_omega} and a global maximizer  $\widehat{\omega}$ is found.
Then the subspace is expanded so that the certain Hermite interpolation properties 
are satisfied at optimal $\widehat{\omega}.$ 
Assuming that $\{\omega_k\}$ converges to a maximizer $\omega_*$ of $\mu(\omega),$ in the next section we show that under some mild assumptions the rate of convergence of Algorithm \ref{alg1} is quadratic. The next result concerns the Hermite interpolation properties between the original and the reduced problems.

\begin{lemma}\label{thm:basic_int}
The following assertions hold regarding Algorithm~\ref{alg1} for each $j = 1,\,\ldots,\,k$:\vspace{1ex}
\begin{enumerate}
    \item[\bf (i)] For all $\gamma \in\left(0,1\right],$ 
   \[
     \sigma_i(\omega_{j},\gamma ) = \sigma_i^{ {\mathcal V}_k}(\omega_{j},\gamma), 	\quad \quad \text{for} \quad \quad i = 1,2,3.
  \]
     \item[\bf (ii)] It holds that 
     $$\mu(\omega_{j}) = \mu^{\mathcal{V}_k}(\omega_{j}).$$ Furthermore, we have
     $\gamma (\omega_j) = \gamma^{{\mathcal V}_k} (\omega_j)$ 
 \item[\bf (iii)] If  $\sigma_2(\omega_{j}, \gamma(\omega_j))$ is simple, 
 then
\[
    \mu'(\omega_j)	 =\left[\mu^{\mathcal{V}_k}\right]'(\omega_j) 
	\] 
	\item[\bf (iv)] If  $\sigma_2(\omega_{j}, \gamma(\omega_j))$ is simple, 
 then
\[
    \mu''(\omega_j)	 =\left[\mu^{\mathcal{V}_k}\right]''(\omega_j) 
	\]

\end{enumerate}
\end{lemma}

\begin{proof}
\noindent \begin{enumerate}
\item[(i)] 
From the lines \ref{init_subspaces}, \ref{later_subspaces}, \ref{orth_subspaces} and Theorem \ref{thm:Gal_interpolate}
we have $H(\omega_j) = H^{\cV_k} (\omega_j)$ and hence
\begin{equation}\label{eq:Tequality1}
T({\omega_j}, \gamma) = T^{\cV_k} ( {\omega_j},\gamma),
\end{equation}
for all $\gamma\in(0,1],$ from which the result follows immediately.
\item[(ii)] 
Observe that, for all $\gamma\in(0,1],$
$$\sigma_2(\omega_j, \gamma(\omega_j) )\leq \sigma(\omega_j,\gamma) \quad \quad \text{and} \quad \quad \sigma_2^{ {\mathcal V}_k}(\omega_j,\gamma^{\cV_k} (\omega_j) )\leq \sigma^{ {\mathcal V}_k}(\omega_j,\gamma).
$$ 
Combining these with part (i) we get
\begin{align*}\label{equality1}
\mu(\omega_j) &= \sigma_2(\omega_j,\gamma (\omega_j) ) = \sigma_2^{{\mathcal V}_k}(\omega_j,\gamma(\omega_j))  \\ 
& \geq  \mu^{ {\mathcal V}_k}(\omega_j) = \sigma_2^{ {\mathcal V}_k}(\omega_j,\gamma^{ {\mathcal V}_k} (\omega_j) ) \\
&= \sigma_2(\omega_j,\gamma^{{\mathcal V}_k} (\omega_j) ) \geq \sigma_2(\omega_j,\gamma (\omega_j) ) \\
&= \mu(\omega_j)
 \end{align*}
Hence, the inequalities above can be replaced by equalities. So, we have $\mu(\omega_j) = \mu^{ {\mathcal V }_k}(\omega_j).$ 
Furthermore, since the function to be minimized in \eqref{eq:mu} is unimodal, we have 
$\gamma (\omega_j) = \gamma^{{\mathcal V}_k} (\omega_j)$ for each $j\in\{1,\cdots,k\}.$
%
\item[(iii)] Suppose that $\sigma_2(\omega_j, \gamma(\omega_j))$ is simple, for a fixed $j\in\{1,\cdots,k\}.$ 
Then, $\mu(\omega)$ and $\mu^{ {\mathcal V}_k }(\omega)$ are differentiable at $\omega=\omega_j.$ Since $T(\omega_j,\gamma(\omega_j)) = T^{\cV_k} (\omega_j,\gamma(\omega_j) ) $ due to \eqref{eq:Tequality1}, the left and the right singular vectors corresponding to  
$\sigma_2(\omega_j,\gamma(\omega_j)) = 
\sigma_2^{{\mathcal V}_k}(\omega_j,\gamma(\omega_j))
$ 
are the same. Let us denote them by $u_2$ and $v_2,$ respectively, and assume, without loss of generality, that they are unit vectors. From the lines \ref{init_subspaces}, \ref{later_subspaces}, \ref{orth_subspaces} together with Theorem \ref{thm:Gal_interpolate} we have

\begin{equation} \label{eq:Tequality2}
\frac{\partial T}{\partial \omega} \left( {\omega_j}, \gamma \right) = 
\frac{\partial T^{\cV_k}}{\partial \omega} \left( {\omega_j},\gamma \right),
\end{equation}

Now, exploiting the analytical formulas for the derivatives of singular value functions \cite{Lancaster1964, Bunse-Gerstner1991} and employing \eqref{eq:Thermite2} we obtain
\begin{align*}
\mu' (\omega_j) &= 
\frac{ \partial \sigma_2}{\partial \omega} \left( \omega_j, \gamma(\omega_j) \right) =
\Real \left( u_2^T \;  \frac{\partial T}{\partial \omega} \left( \omega_j, \gamma(\omega_j) \right) \; v_2 \right) 
\\ & = 
\Real \left( u_2^T \;  \frac{\partial T^{\cV_k} }{\partial \omega} \left(\omega_j, \gamma(\omega_j) \right) \; v_2 \right) = 
\frac{\partial \sigma^{ {\mathcal V}_k} }{\partial \omega} \left(\omega_j, \gamma(\omega_j) \right)
\\
& =  \frac{\partial \sigma^{ {\mathcal V}_k} }{\partial \omega} \left(\omega_j, \gamma^{\cV_k}(\omega_j) \right)
= \left[\mu^{{\cV}_k}\right]' (\omega).
\end{align*}
\item[(iv)]
It follows from the similar arguments as in part (iii), only here we use the analytical formulas for the second derivatives of singular value functions and exploit 

$$\frac{\partial^2 T}{\partial \omega^2} \left( {\omega_j}, \gamma \right) =  \frac{\partial^2 T^{\cV_k}}{\partial \omega^2} \left( {\omega_j},\gamma \right), \quad \forall \gamma $$

\end{enumerate}
\end{proof}

\begin{remark}
Although \eqref{eq.outer} is a maximin problem, when we expand the subspace we only utilize the optimizer of the outer iteration, namely optimal $\omega,$ since $H$ depends only on $\omega.$ Therefore, the problem turns into the maximization of $\mu(\omega)$ and the Hermite interpolation properties between the original problem and the reduced one at only optimal $\omega$ for the reduced problem suffices to get a quadratic convergence as we shall see in the next section.
\end{remark}

\section{Rate of Convergence Analysis}\label{sec:roc}
This section is devoted to the rate of convergence analysis of the subspace framework discussed in the previous section. The main theorem states that the rate of convergence of Algorithm ~\ref{alg1} is quadratic, assuming that it converges at least locally. In particular, we consider the three consecutive iterates $\omega_{k-1}, \omega_k, \omega_{k+1}$ of Algorithm \ref{alg1} that are sufficiently close to a local or global maximizer $\omega_*$ of $\mu(\cdot).$
Before proceeding, we discuss some particular cases regarding the location of optimizers of the inner and outer optimization problems
which were analyzed in detail in \cite{Qiu1995}.

\hspace{2ex}

\noindent \textbf{Case 1: }$\omega_* = 0$. It turns out that, $\omega_* =0$ if and only if
$\Imag(H(\ri\omega_*)) = 0.$ 
Therefore, in this case the the inner minimization over $\gamma$ disappears. More generally, if $\rank(\Imag(M))\leq 1,$ then the minimization over $\gamma$ can be eliminated and $\mu(\cdot)$ can be calculated by means of an explicit formula \cite{Qiu1995}. In particular, when $w_*= 0,$ we have
\begin{align*}
r_\R(A;B,C) &= \sigma_2 \left(
\begin{bmatrix}
\Real( H(\ri\omega_*) ) & 0 \\ 0 & \Real(H(\ri\omega_*))
\end{bmatrix}\right) \\ 
&= \sigma_1 \left(
\begin{bmatrix}
\Real(H(\ri\omega_*)) & 0 \\ 0 & \Real(H(\ri\omega_*))
\end{bmatrix}\right) \\
&= \sigma_1(CA^{-1}B).
\end{align*}

\hspace{2ex}

\noindent \textbf{Case 2:} $\gamma_* =1,$ where and throughout $\gamma_*$
denotes the global minimizer of $\sigma_2(\omega_*,\cdot)$ over $(0,1],$ that is, 
$\gamma_* = \arg\min_{(0,1]}\sigma_2(\omega_*,\gamma).$
In this case, 
$$
\mu \left(H(\ri\omega_*)) \right) = \sigma_2(T(\omega_*,1)) = \sigma_1(T(\omega_*,1)) = \sigma_1(H(\ri\omega_*)).
$$
Hence, the computation of $r_\R(A;B,C)$ becomes equivalent to the computation of $\cH_\infty$ norm, which is already addressed in \cite{Aliyev2017}.

\hspace{2ex}

We remark that in practice, our method efficiently works for the non-smooth cases above (see Section \ref{sec:nr}). However,  the rate of convergence analysis discussed in this section does not cover these cases. To this end, throughout this section we assume that the cases 1 and 2 mentioned above are excluded. That is, we assume that, $\omega_*\neq 0,$ $\gamma_*\neq 1.$ 
More generally, in the rate of convergence analysis we address the smooth setting only. In other words, we assume in the rest of this section that the following smoothness assumption holds. 

\begin{assumption}[Smoothness]\label{ass:nonsm} The singular value $\sigma_2(\omega_*, \gamma_*)$ of $T(\omega_*, \gamma_*)$ is simple.
\end{assumption}
The rate of convergence analysis for 
the non-smooth optimization of the largest eigenvalue function of a Hermitian matrix 
that depends on one parameter is addressed in \cite{Mengi2018}. As a future work, the idea of \cite{Mengi2018} can be adapted to the non-smooth cases of the problem of this work. 

\hspace{2ex}

\noindent \textbf{Case 3: }The minimum of $\sigma(\omega,\cdot)$ is not attained
in $(0,1].$ That is, 
$$
\mu \left(H(\ri\omega) \right) = \inf_{ \gamma \in (0,1]} \sigma_2(\omega,\gamma) = \lim_{\gamma\rightarrow 0} \sigma_2(\omega,\gamma).
$$
It is also shown in \cite{Qiu1995} that this case is possible if and only if $\rank(\Imag(H(\ri\omega)) ) = 1.$ This holds, for example, when $m=1$ or $p=1.$ From \cite{Qiu1995}, we know that in this case
$$
\mu \left( H(\ri\omega) \right) = \max\left\{\sigma_1\left( U_2^T \Real( H (\ri\omega) ) \right) , \sigma_1 \left( \Real ( H(\ri\omega) ) V_2 \right) \right\},
$$
where 
$$
\begin{bmatrix}
U_1 & U_2
\end{bmatrix}
\begin{bmatrix}
\sigma_1( \Imag(H(\ri\omega)) ) & 0 \\ 0 & 0
\end{bmatrix}
\begin{bmatrix}
V_1 & V_2
\end{bmatrix}^T
$$
is a real singular value decomposition of $\Imag(H(\ri\omega)).$

\begin{remark}
The subsequent arguments of this section, in particular, the Lipschitz continuity and boundedness of the derivatives of the singular value functions, does not apply when $\gamma\to 0.$ To cope with this difficulty, throughout this section, we assume that the domain of the inner minimization problem is $\Omega: = [\Gamma,1],$ for some small number $\Gamma$ (in practice we set $\Gamma= 10^{-8}$).
\end{remark}
Results of this section are uniform over all subspaces $\cV_k$ and orthonormal bases $V_k$ as long as the satisfy they following nondegeneracy assumption.

\begin{assumption}[Non-degenracy]\label{ass:nondeg}
For a given $\beta>0,$ we have 
\begin{equation}\label{eq:nondeg}
\sigma_{\min} \left(D^{\cV_k}(\ri\omega_*)   \right)\geq\beta
\end{equation}
where $D^{\cV_k}(s)  := 
(sI-A^{V_k})$ and $\sigma_{\min}(\cdot)$ denotes the smallest singular value of its matrix argument.
\end{assumption}

Assumption \ref{ass:nondeg} combined with the Lipschitz continuity of $\sigma_{\min} (\cdot)$ implies the boundedness of the smallest singular value in \ref{eq:nondeg} away from zero in a neighborhood of $\omega_*.$ Formally, there exists a neighborhood $\mathcal N (\omega_*)$ of $\omega_*$ such that
\begin{equation}\label{eq:sigma_min_bound}
\sigma_{\min} \left(D^{\cV_k}(\ri\omega)   \right) \geq \beta/2,\quad \forall \omega \in \mathcal N (\omega_*),
\end{equation}
see the beginning of Lemma A.1 in \cite{Aliyev2020}.

The next result concerns the uniform Lipschitz continuity of the singular value functions $\omega\mapsto \sigma^{\cV_k} (\omega,\cdot).$ 
For the proofs of the assertions (i) and (ii) we refer to \cite[Lemma A.1]{Aliyev2020}. The proof of the assertion (iii) follows from part (ii); see \cite[Lemma 8, (ii)]{Mengi2018}. Here we make use of the following notations
$$\mathcal{I}(\omega_*, \delta) :=  (\omega_* - \delta, \omega_* + \delta),
\quad 
\mathcal{I}(\gamma_*, \delta) :=  (\gamma_* - \delta, \gamma_* + \delta)
$$
and
$$
\overline{\mathcal{I}} (\omega_*, \delta):=
[\omega_* - \delta, \omega_* + \delta] 
\quad 
\overline{\mathcal{I}}(\gamma_*, \delta):=
[\gamma_* - \delta, \gamma_* + \delta]
$$ 
for given $\omega\in\R, \gamma\in\Omega$ and $\delta>0.$

\noindent By a constant, here and in the rest of this section, we mean a scalar that may depend only on the parameters of the original problem and independent of the subspace $\cV_k$ and orthonormal basis $V_k$ for the subspace.

\begin{lemma}[Uniform Lipschitz Continuity]\label{lemma:lips}
Suppose that Assumption \ref{ass:nondeg} holds.
There exist constants $\delta_\omega, \zeta$ that 
satisfy the following:
\begin{itemize}
\item[\bf (i)] For all $\gamma\in\Omega,$ 
$$
\norm{T^{\cV_k}(\widetilde{\omega}, \gamma ) - T^{\cV_k}(\widehat{\omega}, \gamma ) }_2 \leq \zeta |\widetilde{\omega} - \widehat{\omega}| \quad \forall \widetilde{\omega}, \widehat{\omega} \in\overline{\mathcal{I}} (\omega_*, \delta_\omega).
$$

\item[\bf (ii)] For all $\gamma\in\Omega,$ and $i = 1,2,3,$
$$
\left| \sigma_i^{\cV_k}(\widetilde{\omega}, \gamma ) - \sigma_i^{\cV_k}(\widehat{\omega}, \gamma ) \right|  \leq \zeta |\widetilde{\omega} - \widehat{\omega}| 
\quad \forall \widetilde{\omega}, \widehat{\omega} \in \overline{\mathcal{I}} (\omega_*, \delta_\omega).
$$

\item[\bf (iii)]

$$
\left |\mu^{\cV_k}(\widetilde{\omega}) - \mu^{\cV_k}(\widehat{\omega}) \right | \leq \overline{\zeta} |\widetilde{\omega} - \widehat{\omega}|
 \quad \forall \widetilde{\omega}, \widehat{\omega}
\in \overline{\mathcal{I}} (\omega_*, \delta_\omega).
$$
\end{itemize}
\end{lemma}
Since our main result relies on the smooth setting, in the next result, we state and prove that, there exist an interval in which $\mu(\cdot)$ and $\mu^{\cV_k}(\cdot)$ are both smooth under some mild assumptions. 

\begin{lemma}\label{thm:local_simple}
 Suppose that Assumptions \ref{ass:nonsm} and \ref{ass:nondeg} hold.
\begin{enumerate}
\item[\bf (i)]
There exist constants $\widetilde{\delta}_\omega>0, \widetilde{\delta}_{\gamma}>0$ such that
   both $\sigma_2(\omega, \gamma)$ and $\sigma_2^{{\mathcal V_k}}(\omega,\gamma)$ are simple, hence real analytic, for all 
   $\omega \in \mathcal I (\omega_*, \widetilde{\delta}_\omega)$ 
   and $\gamma \in \mathcal I (\gamma_*, \widetilde{\delta}_\gamma).$
\item[\bf (ii)] 
There exists a constant ${\delta}_\omega >0$ such that
both $\mu(\omega)$ and $\mu^{\cV_k} (\omega)$ are real analytic for all 
 $\omega \in \mathcal I (\omega_*, {\delta}_\omega).$
\end{enumerate}
\end{lemma}

\begin{proof}
\begin{enumerate}
\item[\bf (i)]
Since $A$ is stable, $H$ is analytic on the imaginary axis and so 
$(\omega, \gamma) \mapsto T(\omega,\gamma)$ is continuous on $\R\times \Omega.$ 
Consequently, $\sigma_1(\cdot,\cdot), \sigma_2(\cdot,\cdot), \sigma_3(\cdot,\cdot)$ are continuous functions. The continuity of $\sigma_1(\cdot,\cdot), \sigma_2(\cdot,\cdot), \sigma_3(\cdot,\cdot)$ 
and the assumption that $\sigma_{2} (\omega_*, \gamma_*)$ is simple imply that $\sigma_{2} (\omega,\gamma)$ remains simple
in a neighborhood $\mathcal N(\omega_*, \gamma_*)$ of $(\omega_*, \gamma_*).$
More precisely, 
\begin{eqnarray}\label{eq:simplicity1}
\sigma_1(\omega, \gamma) -  \sigma_2(\omega, \gamma) \geq \widetilde{\epsilon} \quad \text{and} \quad 
\sigma_2(\omega, \gamma) -  \sigma_3(\omega, \gamma) \geq \widetilde{\epsilon} 
\end{eqnarray}
for all $(\omega,\gamma) \in\mathcal N(\omega_*,\gamma_*),$ for some $\widetilde{\epsilon} >0.$
This shows that $\sigma_2(\omega,\gamma)$ is simple in 
$\mathcal N (\omega_*, \gamma_*)$.

Now, by exploiting the interpolation properties
$$
\sigma_i (\omega_k,\gamma) = \sigma^{\cV_k} (\omega_k,\gamma)
$$
for $i=1,2,3$ and $\gamma\in \Omega,$ as well as the uniform Lipschitz
continuity of $\sigma_i^{\cV_k} (\cdot,\cdot),$ for $i=1,2,3,$ and assuming without loss of generality that $\omega_k$ is close enough to $\omega_*$
we deduce that there exists a region $\mathcal I(\omega_*, \widetilde{\delta}_\omega ) \times \mathcal I(\gamma_*, \widetilde{\delta}_\gamma ) \subseteq \mathcal N (\omega_*, \gamma_*)$ in which $\sigma_2^{\cV_k} (\cdot,\cdot)$ is simple. That is,
\begin{equation} \label{eq:simplicity2}
\sigma_1^{\cV_k}(\omega, \gamma) -  \sigma_2^{\cV_k}(\omega, \gamma) \geq {\epsilon}, \;\;  
\sigma_2^{\cV_k}(\omega, \gamma) -  \sigma_3^{\cV_k}(\omega, \gamma) \geq {\epsilon} 
\end{equation}
for all $(\omega,\gamma) \in \mathcal I(\omega_*, \widetilde{\delta}_\omega ) \times \mathcal I(\gamma_*, \widetilde{\delta}_\gamma ) $
and for some constants ${\epsilon} \in (0,\widetilde{\epsilon}), 
\widetilde{\delta}_\omega>0, \widetilde{\delta}_\gamma>0.$
Therefore, $\sigma_2 (\cdot, \cdot)$ and $\sigma_2^{\cV_k} (\cdot, \cdot)$ are simple singular value functions of $T(\omega,\gamma)$ and  $T^{\cV_k}(\omega,\gamma)$ in 
$\mathcal I(\omega_*, \widetilde{\delta}_\omega ) \times \mathcal I(\gamma_*, \widetilde{\delta}_\gamma ).$ 
It follows that 
both $\sigma_2 (\cdot, \cdot)$ and $\sigma_2^{\cV_k} (\cdot, \cdot)$ are real analytic in $\mathcal I(\omega_*, \widetilde{\delta}_\omega ) \times \mathcal I(\gamma_*, \widetilde{\delta}_\gamma ). $

\item[\bf (ii)]
From the implicit function theorem, there exist constants $\delta_\omega>0, \delta_\gamma>0$ such that
for all $\omega\in \mathcal I(\omega_*, {\delta}_\omega ,)$ 
we have
$\gamma(\omega) \in \mathcal I(\gamma_*, {\delta}_\gamma ).$
Now, assume without loss of generality that $\delta_\omega\in(0, \widetilde{\delta}_\omega)$ and 
$\delta_\gamma \in(0, \widetilde{\delta}_\gamma),$ where $\widetilde{\delta}_\omega, \widetilde{\delta}_\gamma $ are as in \eqref{eq:simplicity2}.
Then for $ \omega\in \mathcal I(\omega_*, {\delta}_\omega ) \subset   \mathcal I(\omega_*, \widetilde{\delta}_\omega )
 $ 
we have
$\gamma(\omega)  \in \mathcal I(\gamma_*, {\delta}_\gamma )
\subset \mathcal I(\gamma_*, \widetilde{\delta}_\gamma ).$ 
Hence, \eqref{eq:simplicity2} implies 
\begin{equation} 
\sigma_1\left( \omega, \gamma(\omega) \right) -  \sigma_2 \left(\omega, \gamma(\omega) \right) \geq {\epsilon}\quad \text{and} \quad 
\sigma_2 \left( \omega, \gamma(\omega) \right) -  \sigma_3 \left(\omega, \gamma(\omega) \right) \geq {\epsilon}  
\end{equation}
for all $\omega \in \mathcal I  \left(\omega_*, {\delta}_\omega \right).$
Thus, $\sigma_2 \left( \omega, \gamma(\omega) \right)$ is simple 
and so $\mu(\omega)$ is real analytic in $ \mathcal I  \left(\omega_*, {\delta}_\omega \right).$
Considering $\omega_k$ close enough to $\omega_*$ and following the similar arguments as above we conclude that 
$\mu^{\cV_k}(\omega)$ is also real analytic in $ \mathcal I  \left(\omega_*, {\delta}_\omega \right).$
\end{enumerate}
\end{proof}
The next result concerns the uniform boundedness of the third derivatives of $\mu(\cdot)$ and $\mu^{\cV_k}(\cdot)$ in a neighborhood of $\omega_*.$ This uniform boundedness of the third derivatives gives rise to the uniform Lipschitz continuity of the second derivatives of these functions, which becomes important in our main result. For a proof we refer to Lemma A.2 (ii) and Proposition 3.5 in \cite{Aliyev2020}

\begin{lemma}\label{lemma:boundedness}
Suppose that Assumptions \ref{ass:nonsm} and \ref{ass:nondeg} hold. 
\begin{enumerate}
\item[\bf(i)]
Both $\mu(\cdot)$ and $\mu^{\cV_k}(\cdot)$ are at least three times continuously differentiable in $\mathcal I (\omega_*,\delta_\omega),$ where $\delta_\omega$ is as in Lemma ~\ref{thm:local_simple}.

\item[\bf(ii)]
For each $\hat{\delta}_{\omega}\in (0, \delta_\omega)$ there exists a constant $\xi>0$ such that for all $\omega\in \mathcal I (\omega_*, \hat{\delta}_{\omega})$ 
we have
$$\left|\mu'''(\omega)\right|\leq \xi \quad \text{and} \quad \left|\left[\mu^{\cV_k}\right]'''(\omega)\right|\leq \xi.
$$


\end{enumerate}
\end{lemma}
Now, we are ready to present our main result. Here, we remind that $\omega_{k-1}, \omega_k, \omega_{k+1}$ are assumed to be sufficiently close to a local maximizer $\omega_*$ of $\mu(\cdot).$

\begin{theorem}[Local Quadratic Convergence]\label{thm:quad_conv}
Suppose that Assumptions \ref{ass:nonsm} and \ref{ass:nondeg} hold and $\mu''(\omega_*) \neq 0.$ 
Then, for Algorithm \ref{alg1}, there exists a constant $c>0$ such that
 
$$ 
\frac{ | \omega_{k+1} - \omega_\ast |}{| \omega_k - \omega_\ast |^2}
 \leq  c 
$$
\end{theorem}

\begin{proof}
Assume without loss of generality that $\omega_k, \omega_{k-1}$ lie in 
$\mathcal I (\omega_*, \delta_{\omega} )$ and are close enough to $\omega_*$ so that
$\mathcal I (\omega_k, h_k )\subset \mathcal I (\omega_*, \delta_{\omega} ),$ where $h_k := |\omega_k - \omega_{k-1} |$ and $\delta_{\omega}$ is as in Lemma \ref{thm:local_simple} and Lemma \ref{lemma:boundedness}.
Assume further that $\mu''(\omega_k) \neq 0,$ 
that is, 
\begin{equation}\label{eq:non_zero}
| \mu''(\omega_k) | \geq l
\end{equation}
 for some constant $l>0.$ Existence of such $l>0$ 
follows from the assumption $\mu''(\omega_*) \neq 0 $ and the continuity of $\mu''(\cdot)$ in $\mathcal I (\omega_*, \delta_{\omega} ).$
Due to analyticity of $\mu(\cdot)$ in $\mathcal I (\omega_*, \delta_{\omega} )$ (Lemma \ref{thm:local_simple}, part (ii))
we have 
\begin{equation}\label{eq:taylor}
0 = \mu'(\omega_*) = \mu'(\omega_k) + \int_0^1 \mu'' \left( \omega_* + t ( \omega_* - \omega_k) \right) (\omega_* - \omega_k)\mathrm{d}t.
\end{equation}
Dividing both sides of \eqref{eq:taylor} by $\mu''(\omega_k)$ we obtain
\begin{equation}\label{eq:taylor_organized}
	0 =  \frac{\mu'(\omega_k)}{\mu''(\omega_k)} + \left(\omega_\ast - \omega_k \right) +
	\frac{1}{\mu''(\omega_k)}
			\int_0^1 \left(	\mu'' \left(\omega_k + t (\omega_\ast - \omega_k )\right) - \mu''(\omega_k) 	\right)
							\left(\omega_\ast - \omega_k \right) \mathrm{d}t.
\end{equation}
Application of Taylor's theorem with integral remainder to $[\mu^{\cV_k}]' (\cdot)$ at $\omega_k,$ and optimality of $\omega_{k+1}$ with respect to 
$\mu^{\cV_k}(\cdot)$ give rise to
\begin{equation}\label{eq:taylor_red}
0 = [\mu^{\cV_k}]'(\omega_{k+1} ) = [\mu^{\cV_k}]'(\omega_k) + \int_0^1 [\mu^{\cV_k}]'' \left( \omega_k + t ( \omega_{k+1} - \omega_k) \right) (\omega_{k+1} - \omega_k)\mathrm{d}t.
\end{equation}
which implies
\begin{equation}\label{eq:taylor_organized_red}
\begin{split}
	0 &=  \frac{ [\mu^{\cV_k}]'(\omega_k)}{[\mu^{\cV_k}]''(\omega_k)} + \left(\omega_{k+1} - \omega_k \right) \\ 
	&+	\frac{1}{[\mu^{\cV_k}]''(\omega_k)}
			\int_0^1 \left(	[\mu^{\cV_k}]'' \left(\omega_k + t (\omega_{k+1} - \omega_k )\right) - [\mu^{\cV_k}]''(\omega_k) 	\right)
							\left(\omega_{k+1} - \omega_k \right) \mathrm{d}t.
\end{split}						
\end{equation}
Combining \eqref{eq:taylor_organized} and \eqref{eq:taylor_organized_red} and exploiting the Hermite interpolation properties 
$$
\mu'(\omega_k) = [\mu^{\cV_k}]'(\omega_k)
\quad
\text{and} \quad
\mu''(\omega_k) = [\mu^{\cV_k}]''(\omega_k)
$$
we deduce
\begin{multline}
0 = \omega_* - \omega_{k+1} + \frac{1}{\mu''(\omega_k)} 
 \int_0^1 \left(	\mu'' \left(\omega_k + t (\omega_\ast - \omega_k )\right) - \mu''(\omega_k) 	\right)
							\left(\omega_\ast 
							- \omega_k \right) \mathrm{d}t \\
							- \frac{1}{\mu''(\omega_k)}  \int_0^1 \left(	[\mu^{\cV_k}]'' \left(\omega_k + t (\omega_{k+1} - \omega_k )\right) - [\mu^{\cV_k}]''(\omega_k) 	\right)
							\left(\omega_{k+1} - \omega_k \right) \mathrm{d}t 
\end{multline}
implying
\begin{multline*}
\left | \omega_* - \omega_{k+1} \right |  \leq \left | \frac{1}{\mu''(\omega_k)} \right | 
\left | \int_0^1 \left( \mu'' \left(\omega_k + t (\omega_\ast - \omega_k )\right) - \mu''(\omega_k) \right)\left(\omega_\ast - \omega_k \right) \mathrm{d}t \right |   \\
              +  \left | \frac{1}{\mu''(\omega_k)} \right | \left | \int_0^1 \left( [\mu^{\cV_k}]'' \left(\omega_k + t (\omega_{k+1} - \omega_k )\right) - [\mu^{\cV_k}]''(\omega_k) \right)\left(\omega_{k+1} -\omega_k \right) \mathrm{d}t \right | 
\end{multline*}
In the last equation, we exploit the Lipschitz continuity of both $\mu''(\cdot)$ and $\left[ \mu^{\cV_k} \right]''(\cdot)$ in $\mathcal I (\omega_*, \delta_\omega)$ (which follows from Lemma \ref{lemma:boundedness})  and employ \eqref{eq:non_zero} to obtain
\begin{equation}\label{ineq1}
\left | \omega_* - \omega_{k+1} \right |  \leq  \xi/l  \left( |\omega_* - \omega_k |^2 +  |\omega_{k+1} - \omega_k |^2 \right) 
\end{equation}
where $\xi$ is the Lipschitz constant for $\mu''$ and $[\mu^\cV_k]''.$
The result follows by substituting the inequality
$$
|\omega_{k+1} - \omega_k |^2 \leq 2 |\omega_{k+1} - \omega_* |^2 + 2|\omega_{k} - \omega_* |^2
$$
in \eqref{ineq1}.
\end{proof}

\section{Numerical experiments}\label{sec:nr}
In this section, we present numerical results obtained by our MATLAB implementation of Algorithm~\ref{alg1}. 
First we introduce some important implementation details and the test setup. 
Then, we report the numerical results on several large-scale examples taken from Model Order Reduction Wiki (MOR Wiki) website\footnote{available at \url{https://morwiki.mpi-magdeburg.mpg.de/morwiki/index.php/Main_Page}}, EigTool\footnote{available at \url{https://www.cs.ox.ac.uk/pseudospectra/eigtool/}} and Compleib \footnote{available at \url{http://www.complib.de}}. In all examples the system matrices $(A,B,C)$ are real and $A$ is asymptotically stable.
Our numerical experiments have been performed using MATLAB (R2018b) on an iMac with an Intel Core i5 1.4 GHz processor and 4 GB memory

\subsection{Implementation Details and Test Setup}
At each iteration of Algorithm~\ref{alg1},
a reduced problem is solved, namely, $r^{\cV}(A;B,C)$ is calculated. Therefore, at each iteration we solve a small scale minimax problem. Since the cost function of the inner minimization problem is unimodal \cite{Qiu1995}, we utilize ``golden section search algorithm" to solve the inner minimization problem. In fact, we use MATLAB  function \texttt{fminbnd} to solve the inner minimization problem. The outer maximization problem is solved by means of \texttt{eigopt}, a MATLAB implementation of the algorithm in \cite{Mengi2014}. In this MATLAB package a lower and an upper bound for the optimal value of a given eigenvalue function is found by employing piece-wise quadratic support functions, and the algorithm terminates when the difference between these bounds is less than a prescribed termination tolerance. In our experiments the termination tolerance equals to $10^{-6}.$  
Algorithm~\ref{alg1} terminates in practice when the relative distance between 
$r^{\cV_k}(A;B,C)$ and $r^{\cV_{k-1}}(A;B,C)$
is less than a prescribed tolerance for some $k> 1$, or if the number of iterations exceeds a specified integer. Formally, the algorithm terminates, if
\begin{multline*}
	k >  k_{\max} \quad \text{or} \quad
	\left| r^{\cV_k}(A;B,C) -  
					 r^{\cV_{k-1}}(A;B,C) \right|
					 < 
			\varepsilon \cdot \frac{1}{2} \left| r^{\cV_k}(A;B,C) +  r^{\cV_{k-1}}(A;B,C) \right|.
\end{multline*}
In our numerical experiments, we set 
$\varepsilon = 10^{-4}$ and $k_{\max} = 15$.

\subsection{Test Results}
In Table \ref{tab:results_nonsmooth} and Table \ref{tab:results_smooth} we report the outcome of our numerical experiments. The number of additional iterations after the construction of the initial reduced function needed to return the $r_\R(A;B,C)$ is denoted by $n_{iter}$ in both tables. Furthermore, the dimension of $A,$ the number of the columns of $B$ and the number of the rows of $C$ are denoted by $n,m,p,$ respectively. 
\begin{table}[h]
 \caption{Numerical results of test examples for which $r_{\R}(A;B,C)$ is attained at $\omega_* = 0:$ (top) the results of Algorithm ~\ref{alg1} are compared with the results of the method in \cite{Guglielmi2017}; (bottom) the number of subspace iterations and runtimes required by 
\cite{Guglielmi2017} and that required by Algorithm \ref{alg1} are listed.
 } 
 \label{tab:results_nonsmooth}
\begin{center}
\begin{tabular}{c|ccc|cc}
 \hline 
  & & & &      \multicolumn{2}{c}{ computed $r_{\R(A;B,C)}$ }  \\ [0.5ex]
     Example & $n$ & $m$ & $p$  & Algor.~\ref{alg1} & \cite{Guglielmi2017}  \\ [0.5ex]
  \hline 
     \texttt{HF2D3}  		& 4489   &  2  & 4      &  6.87987e$-$01  & 6.87987e$-$01     \\
     \texttt{HF2D4}   		& 2025   &  2  & 4    &  2.78245e$-$02  & 2.78245e$-$02      \\
    \texttt{Thermal Block} & 7488 & 1 & 4  & 9.74509e$-$01    &  9.74509e$-$01  \\
    \texttt{Filter2D} & 1668 & 1 & 5  & 1.88560e$-$03    &  1.88561e$-$03  \\
    \texttt{Skewlap3D} & 1331 & 4 & 5  & 6.89850e$+$00    &  6.89850e$+$00 
   \\
    \texttt{Convec Diff} & 3600 & 1 & 1  & 3.45182e$-$01    &  3.45182e$-$01  \\
   \end{tabular}
   \end{center}

\begin{center}
\begin{tabular}{c|c|cc|c}
 \hline 
  &   &   \multicolumn{3}{c}{Time in s} \\ [0.5ex]
     Example & $n_{\rm iter}$ & Algor.~\ref{alg1} &  \cite{Guglielmi2017}  & ratio \\ [0.5ex]
  \hline 
     \texttt{HF2D3}  & 2	 &    0.75   &    56.01  & 74.68 \\
 \texttt{HF2D4}   	& 2	     &  0.71      & 11.42  & 16.08 \\
 \texttt{Thermal Block} & 3 & 10.05 & 109.05 & 
   10.85 \\
    \texttt{Filter2D}  & 2 & 1.27 & 13.88 & 
   11.02 \\
    \texttt{Skewlap3D} & 1 & 1.10 & 6.45 & 5.86
   \\
    \texttt{Convec Diff} & 2 & 3.13 & 14.19 & 
   4.52 \\
   \end{tabular}
   \end{center}
   \end{table}

\noindent In Table \ref{tab:results_nonsmooth}, we exhibit the results of the examples for which $r_\R(A;B,C)$ is attained at $\omega_* = 0.$ Here, we compare the results of our subspace method with the ones generated by \cite{Guglielmi2017} in which the perturbations are  bounded with respect to the Frobenius norm. The complex structured stability radius $r_\C(A;B,C)$ of each of the examples listed in Table \ref{tab:results_nonsmooth}, 
is also attained at $\omega_* = 0,$ which implies that $r_\R(A;B,C)=r_\C(A;B,C).$ Notice that in the complex case, that is, when $\F=\C,$ the spectral norm and Frobenius norm define the same stability radius $r_\C(A;B,C)$ \cite{Guglielmi2017}. Therefore, comparing the results of Algorithm \ref{alg1} with the ones generated by the approach in \cite{Guglielmi2017} is reasonable. 

\hspace{2ex}

From Table \ref{tab:results_nonsmooth} we see that the correct value of $r_\R(A;B,C)$ is found by Algorithm \ref{alg1} for each of the examples. In terms of the runtime, Algorithm \ref{alg1} outperforms the approach of \cite{Guglielmi2017}. The ratios between the time required by 
\cite{Guglielmi2017} and that required by Algorithm \ref{alg1} are listed in the last column of the bottom row of Table \ref{tab:results_nonsmooth}. 
 
\hspace{2ex} 
 
Recall that, the rate of convergence analysis discussed in Section \ref{sec:roc} relies on the smooth setting, and when $w_*=0,$ the singular value function $\sigma_2(,\cdot,)$ is not simple and hence not smooth. However, the quadratic convergence is still observed for this particular non-smooth case.
For all examples listed in Table \ref{tab:results_nonsmooth}, at most four iterations are required to estimate $r_\R(A;B,C)$ up to a given tolerance. Specifically, for the \texttt{Thermal Block} example, we report the errors in the iterates as a function of the number of iterations in Table \ref{tab:erros_nonsmooth}.
\begin{table}[h]
\caption{For the example \texttt{Thermal Block}, the optimizers $\omega_{k+1},$ the errors $|\mu_k-\mu_*|$ and $|r_k-r_*|$ are listed, where the short-hands 
$\mu_k := \mu^{\cV_k}(\omega_{k+1}),$ 
$\mu_*:= \mu(\omega_*),$ 
$r_k = r^{\cV_k}(A;B,C)$ and 
$r_* = r(A;B,C)$ are used.
As the ``exact'' solution we have taken the one we obtain from Algorithm \ref{alg1} after four iterations.}
\label{tab:erros_nonsmooth}
\begin{center}
\begin{tabular}{l|ccc}
\hline
	 $k$	  & $\omega_{k+1}$ & $|\mu_k - \mu_\ast | $
	 & $|r_k-r_*|$	\\	
\hline	
   0   &  -1.172e$-$02 & 1.026e$+$00 & 2.678e$+$04       \\
   1   &  0 & 7.465$-$01  & 2.601e$+$00                    \\
   2   &  0 & 6.656e$-$11  & 6.321e$-$11      \\
   3   &  0 &   0 & 0   \\
\end{tabular}
\end{center}
\end{table}

In Table \ref{tab:results_smooth}, we summarize the results of another set of examples. Here, the optimal frequency $\omega_*$ is different from zero and so the spectral norm and Frobenius norm may define different values for the real stability radius. For this reason, here we use \texttt{eigopt} \cite{Mengi2014} (for the unreduced problems) to compare our results with. 
\begin{table}[h]
 \caption{Numerical results of test examples for which $r_{\R}(A;B,C)$ is attained at $\omega_* \neq 0:$ (top) the results of Algorithm ~\ref{alg1} are compared with the results of the method in \cite{Mengi2014}; (bottom) the number of subspace iterations and runtimes required by \cite{Mengi2014} and that required by Algorithm \ref{alg1} are listed.} 
 \label{tab:results_smooth}
 \begin{center}
 \begin{tabular}{c|ccc|cc|cc}
 \hline
  & & & &  \multicolumn{2}{c}{$r_{\R}(A;B,C)$} & \multicolumn{2}{c}{Optimal frequency $\omega_*$}  \\
     Example & $n$ & $m$ & $p$ & Algor.~\ref{alg1}  & \cite{Mengi2014} & Algor.~\ref{alg1}  & \cite{Mengi2014}   \\
  \hline 
     \texttt{RCL circuit}  & 1841  &  16  & 16    &  8.26909e$-$02  & 8.26909e$-$02 & 1.34499e$-$01 & 1.34516e$+$00   \\
     \texttt{Iss}  & 1412  &  3  & 3   &  8.73788e$+$01  & 8.73788e$+$01 & 1.32378e$+$01 & 1.32378e$+$01   \\
    \texttt{Synt4K} & 4000 & 1 & 1 &  1.06956e$-$01 & 1.06956e$-01$ & -1.47046e$-$03 & -1.47046e$-$03 \\
   
    \texttt{G600} & 1201 & 1 & 3 & 6.87848e$+$00 & 6.87848e$+00$ & -3.04219e$-$01 & -0.30410e$-$01   \\
    
     \texttt{Demmel3K} & 3000 & 7 & 5  & 2.06428e$-$04 &  2.06428e$-$04 & 8.09805e$+$01 & 8.09805e$+$01   \\
     \texttt{Demmel5K} & 5000 & 2 & 2  & 7.64102e$+$00 &  7.64102e$+$00 & 4.26367e$+$01 & 4.26367e$+$01   \\
\end{tabular} 
\end{center}

\begin{center}
\begin{tabular}{c|c|cc|c}
 \hline
  & &      \multicolumn{2}{c}{Time in s} \\
     Example & $n_{\rm iter}$ & Algor.~\ref{alg1} & \cite{Mengi2014} & ratio \\
  \hline 
     \texttt{RCL circuit}    &  3  &   5.7 & 105.6  & 18.5 \\
     \texttt{Iss}  &  2   & 9.4      & 40.1  & 4.3 
     \\
    \texttt{Synt4K}  & 2  &  9.9 & 450.1 & 45.3 \\
   
    \texttt{G600} &  2  & 1.7 & 124.7 & 72.5 \\
    
     \texttt{Demmel3K}  & 2   & 1.0 & 8.4 & 8.4 \\
    
     \texttt{Demmel5K} & 2   & 4.59  & 26.5 & 5.8 \\
\end{tabular} 
\end{center}
\end{table}

For all examples, with or without subspace acceleration, we retrieve the same $r_\R(A;B,C)$ values up to the prescribed tolerance. 
The proposed subspace framework outperforms \texttt{eigopt} for all examples listed in Table \ref{tab:results_smooth}. The ratios of the runtimes are listed in the last column of the table. Again, we observe quadratic convergence in all examples.  The errors in the iterates as a function of the number of iterations for the \texttt{RCL circuit} example are reported in Table \ref{tab:erros_smooth}.
\begin{table}[t]
\caption{The errors of the iterates of Algorithm \ref{alg1}, the errors $|\mu_k-\mu_*|$ and $|r_k-r_*|$ are listed, for the example \texttt{RCL circuit}. Here the short-hands 
$\mu_k := \mu^{\cV_k}(\omega_{k+1}),$ 
$\mu_*:= \mu(\omega_*),$ 
$r_k = r^{\cV_k}(A;B,C)$ and 
$r_* = r(A;B,C)$ are used. As the ``exact'' solution we have taken the one we obtain by Algorithm \ref{alg1} after five iterations.}
\label{tab:erros_smooth}
\begin{center}
\begin{tabular}{l|ccc}
\hline
	 $k$	  & $|\omega_{k+1}-\omega_k|$ & $|\mu_k - \mu_\ast | $
	 & $|r_k-r_*|$	\\	
\hline	
   0   &  2.771e$-$03 & 4.781e$+$00 & 2.342e$-$02       \\
   1   &  6.664e$-$03  & 2.857$+$00  & 1.580e$-$02                    \\
   2   &  9.632e-06 & 2.422e$-$05  & 1.656e$-$07 \\
   3   &  0 &   0 & 0   \\
\end{tabular}
\end{center}
\end{table}

\section{Conclusion}
We have proposed a subspace framework to approximate the structured 
real stability radius for large-scale systems. The method for the computation of the structured real stability radius was based on an interpolatory model order reduction technique. 
In particular, we reduced the large-scale problems to a small one and by repeatedly solving small-scale problems, we obtain an estimate of the large-scale stability radius. Our subspace method yields a Hermite interpolation property between the full and reduced problems which ensures quadratic convergence. The rigorous analysis of this fast convergence was established as well as the efficiency of the proposed methods was demonstrated on several numerical experiments. 
The subspace idea addressed in this work can be applicable for the computation of the structured real stability radius of a large-scale time-delay system. Moreover, a different variant of the subspace method (see for example, \cite{Mengi2018,Kangal2018}) can be applied to approximate the unstructured real stability radius for large-scale systems. We aim to focus on these problems in future works.


\vskip 2ex

\noindent
\textbf{Acknowledgements.} 
The author is grateful to Emre Mengi for his invaluable feedback.

\bibliography{na_references}

\end{document}